\documentclass[12pt,a4paper]{article}
\pdfoutput=1

\usepackage[utf8]{inputenc}
\usepackage{enumerate}
\usepackage{amsmath}
\usepackage{amsfonts}
\usepackage{amssymb}
\usepackage{amsthm}
\usepackage{mathtools}
\usepackage{tikz}
\usepackage{xspace}
\usepackage{authblk}
\usepackage{paralist}
\usepackage{fullpage}

\tikzset{
  vertex/.style={circle,draw, inner sep = 2pt},
  perm/.style={},
  edge/.style={, color=black},
  curveedge/.style={color=black,out=90,in=90},
}

\usepackage{hyperref}
\usepackage[capitalize]{cleveref}

\usepackage[backgroundcolor=gray!10]{todonotes}
\usepackage[backend=biber,isbn=false,url=true,doi=true,maxcitenames=2,maxbibnames=99]{biblatex}
\bibliography{strings-long,bibliography}

\newtheorem{theorem}{Theorem}[section]

\newtheorem{observation}[theorem]{Observation}
\newtheorem{lemma}[theorem]{Lemma}

\newtheorem{definition}[theorem]{Definition}

\DeclarePairedDelimiterX{\abs}[1]{\lvert}{\rvert}{#1}

\DeclareMathOperator{\id}{id}
\newcommand{\tlgraph}{terrain-like graph\xspace}
\newcommand{\tl}{terrain-like\xspace}
\newcommand{\tlset}{\mathcal{T}_n}
\newcommand{\Dperm}{Dumont permutation of the second kind\xspace}
\newcommand{\Dset}{\mathcal{D}^2_{2n}}
\newcommand{\Dder}{\tilde{\mathcal{D}}^2_{2n}}
\newcommand{\Gset}{\mathcal{G}_n}
\newcommand{\Pset}{\mathcal{S}_{2n}}
\newcommand{\oneto}[1]{[#1]}

\title{Terrain-like Graphs and the Median Genocchi Numbers}

\author[1]{Vincent Froese}
\author[1]{Malte Renken\footnote{Supported by the DFG project NI~369/17-1.}}

\affil[1]{\small
  Algorithmics and Computational Complexity, Faculty~IV, TU Berlin, Berlin, Germany,\protect\\
  \{vincent.froese, m.renken\}@tu-berlin.de}

\begin{document}
\maketitle

\begin{abstract}
  A graph with vertex set~$\{1,\ldots,n\}$ is \emph{terrain-like} if, for any edge pair
  $\{a,c\},\{b,d\}$ with~$a<b<c<d$, the edge~$\{a,d\}$ also exists.
  Terrain-like graphs frequently appear in geometry in the context of visibility graphs.
  We show that terrain-like graphs are counted by the median Genocchi numbers.
  To this end, we prove a bijection between terrain-like graphs and Dumont derangements of the second kind.
\end{abstract}

\section{Introduction}

The Genocchi numbers appear in various areas with several different combinatorial interpretations~\cite{Dumont74,HZ99,Bigeni14,Hetyei19}
The \emph{Genocchi numbers of the first kind} \cite{Genocchi} are
\[G_1=1,\,G_2=1,\,G_3=3,\,G_4=17,\,G_5=155,\ldots\]
and the \emph{Genocchi numbers of the second kind} (also called \emph{median Genocchi numbers})~\cite{medGenocchi} are
\[H_1=1,\, H_2=2,\, H_3=8,\, H_4=56,\, H_5=608,\ldots.\]

Both kinds of Genocchi numbers are closely tied to the following set of permutations.
(Here $\Pset$ denotes the symmetric group on~$[2n] = \{1, 2, \dots, 2n\}$).

\begin{definition}
  A permutation~$\pi\in \Pset$ is a \emph{\Dperm} if~$\pi(2i-1) \ge 2i-1$ and~$\pi(2i) < 2i$ for all~$i\in[n]$.

  The set of all Dumont permutations of the second kind is denoted by~$\Dset$.
\end{definition}

It is known that~$|\Dset|=G_{n+1}$~\cite{Dumont74}.
For the subset~$\Dder\subset \Dset$ of all \emph{derangements}, that is, permutations~$\pi\in\Dset$ with~$\pi(2i-1) > 2i-1$ for all~$i\in[n-1]$, it holds~$|\Dder|=H_n$~\cite{DR94}.
\Cref{fig:examples} shows all permutations in~$\Dder$ for~$n=3$.

Let now~$\Gset$ denote the set of all undirected simple graphs with vertex set~$[n]$.
\begin{definition}
  A graph~$G=([n], E)\in\Gset$ is \emph{terrain-like} if it satisfies the \emph{X-property}, that is, if
  \[ \{\{a, c\}, \{b, d\}\} \subseteq E \implies \{a, d\} \in E \]
  holds for all $a<b<c<d$.
  
  The set of all \tl graphs in $\Gset$ is denoted~$\tlset$.
\end{definition}

Terrain-like graphs~\cite{AFKS22,FR21} often appear in geometry in the context of visibility graphs and contain the class of \emph{persistent graphs} which one-to-one correspond to triangulations of the three-dimensional cyclic polytope~\cite{FR21c} and are a superset of the so-called \emph{terrain-visibility graphs}~\cite{AGKSW20,FR19,Saeedi2015}.

We show that~$|\tlset|=H_n$ by proving the following.

\begin{theorem}
  For every~$n\ge 1$, there is a bijection between~$\tlset$ and~$\Dder$.
\end{theorem}

\Cref{fig:examples} explicitly gives that bijection for $n = 3$. (Note that $\mathcal{T}_3$ contains simply all 3-vertex graphs.)

It is not hard to see that~$(2i,2i+1)\pi\in\Dder$ for every~$\pi\in\Dder$ and~$i\in[n-1]$, which implies that~$H_n$ is divisible by~$2^{n-1}$ (the numbers~$h_n = H_n/2^{n-1}$ are called \emph{normalized median Genocchi numbers}~\cite{normMedGenocchi,HZ99,Bigeni14}).
This can also easily be derived from~$\tlset$ by observing that in any graph~$G\in\tlset$ insertion or deletion of the edge~$\{i,i+1\}$ for~$i\in[n-1]$ preserves the X-property.
We remark that in our bijection the \emph{normalized Genocchi permutations} in~$\Dder$ as defined by~\textcite{HZ99} correspond to the \tl graphs which contain all the edges~$\{i,i+1\}$, $i\in[n-1]$.

\begin{figure}[t]
  \centering
  \begin{tikzpicture}
    \node[vertex] (1) at (0,0) {};
    \node[vertex] (2) at (1,0) {};
    \node[vertex] (3) at (2,0) {};
    \node at (1,-.5) {214365};
    \begin{scope}[shift={(3.5,0)}]
      \node[vertex] (1) at (0,0) {};
      \node[vertex] (2) at (1,0) {};
      \node[vertex] (3) at (2,0) {};
      \node at (1,-.5) {314265};
      \draw (1) -- (2);
    \end{scope}
    \begin{scope}[shift={(7,0)}]
      \node[vertex] (1) at (0,0) {};
      \node[vertex] (2) at (1,0) {};
      \node[vertex] (3) at (2,0) {};
      \node at (1,-.5) {215364};
      \draw (2) -- (3);
    \end{scope}
    \begin{scope}[shift={(10.5,0)}]
      \node[vertex] (1) at (0,0) {};
      \node[vertex] (2) at (1,0) {};
      \node[vertex] (3) at (2,0) {};
      \node at (1,-.5) {514362};
      \draw[curveedge] (1) to (3);
    \end{scope}
    \begin{scope}[shift={(0,-2)}]
      \node[vertex] (1) at (0,0) {};
      \node[vertex] (2) at (1,0) {};
      \node[vertex] (3) at (2,0) {};
      \node at (1,-.5) {315264};
      \draw (1) -- (2) -- (3);
    \end{scope}
    \begin{scope}[shift={(3.5,-2)}]
      \node[vertex] (1) at (0,0) {};
      \node[vertex] (2) at (1,0) {};
      \node[vertex] (3) at (2,0) {};
      \node at (1,-.5) {514263};
      \draw (1) -- (2);
      \draw[curveedge] (1) to (3);
    \end{scope}
    \begin{scope}[shift={(7,-2)}]
      \node[vertex] (1) at (0,0) {};
      \node[vertex] (2) at (1,0) {};
      \node[vertex] (3) at (2,0) {};
      \node at (1,-.5) {415362};
      \draw (2) -- (3);
      \draw[curveedge] (1) to (3);
    \end{scope}
    \begin{scope}[shift={(10.5,-2)}]
      \node[vertex] (1) at (0,0) {};
      \node[vertex] (2) at (1,0) {};
      \node[vertex] (3) at (2,0) {};
      \node at (1,-.5) {415263};
      \draw (1) -- (2) -- (3);
      \draw[curveedge] (1) to (3);
    \end{scope}
  \end{tikzpicture}
  \caption{The~$H_3=8$ permutations in~$\tilde{\mathcal{D}}^2_6$ with the corresponding graphs from~$\mathcal{T}_3$.}
  \label{fig:examples}
\end{figure}
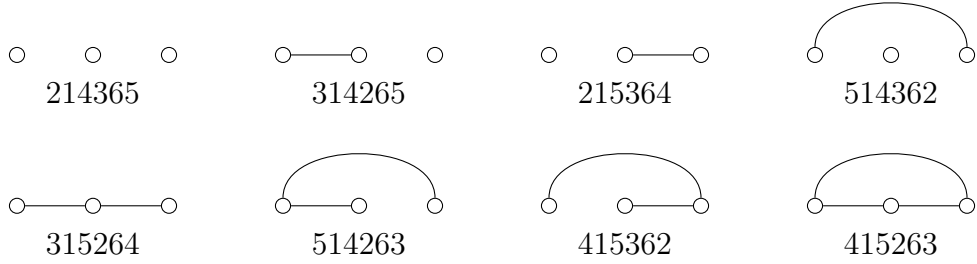

\section{Two maps between $\tlset{}$ and $\Dder$}

We start with a map~$\Pi$ that maps graphs to Dumont derangements of the second kind (notably, this will work for arbitrary graphs).
The underlying idea is to start with a specific Dumont permutation and then apply a sequence of transpositions, with each transposition corresponding to an edge of the graph.
To this end, we define a partial order on vertex pairs (``from inner to outer'').

\begin{definition}
  For~$a< b\in[n]$ and~$c < d\in[n]$, we define
  \[\{a,b\} \preceq \{c,d\} \iff c \leq a < b \leq d.\]
\end{definition}
\noindent
We call any ordering $\leq$ of a subset $E \subseteq \binom{[n]}{2}$ \emph{valid} if $e \preceq e' \implies e \leq e'$ for all $e, e' \in E$.
For a pair $\{a, b\}\in\binom{[n]}{2}$ with $a < b$, let $\tau(\{a, b\})$ denote the transposition $(2a, 2b-1)$.
We now define the map~$\Pi$ as follows.

\begin{definition}
  Let $G=([n],\{e_1 > \ldots > e_m\})$ be a graph with a valid edge ordering.
We define $\Pi\colon \Gset{} \to \Pset{}$ by \[\Pi(G) = \tau(e_m) \tau(e_{m-1}) \dots \tau(e_1) \pi_0,\]
where $\pi_0 \coloneqq (1,2)(3,4)\dots(2n-1,2n)\in \Dder$.
\end{definition}

\Cref{fig:example} depicts an example.

\begin{lemma}\label{thm:welldef}
$\Pi$ is well-defined.
\end{lemma}
\begin{proof}
We have to show that $\Pi(G)$ does not depend on the chosen edge ordering.
To this end, let $e_i = \{a, b\}$ and $e_j = \{c, d\}$ with $a < b$ and $c < d$ be incomparable with respect to~$\preceq$.
Then, $a \neq c$ and $b \neq d$. Therefore, $2a$, $2b-1$, $2c$, and $2d-1$ are pairwise distinct,
and thus, $\tau(e_i) = (2a, 2b-1)$ and $\tau(e_j) = (2c, 2d-1)$ commute.
Repeated application of this argument proves the claim.
\end{proof}

\begin{lemma}
  For every graph~$G\in\mathcal{G}_n$, it holds~$\Pi(G)\in \Dder$.
\end{lemma}

\begin{proof}
  Let $G=([n],E=\{e_1 > \ldots > e_m\})$ be a graph with a valid edge ordering
  and let~$\pi \coloneqq \Pi(G)$.
  We will show that~$\pi(i) < i$ holds for each even~$i\in[2n]$ (the argument for odd~$i$ is analogous).

  Let~$i=2v$, $v\in[n]$, and consider the sequence of transpositions~$\tau(e_1),\ldots,\tau(e_m)$.
  Note that~$\pi_0(i)=i-1$. Hence, if the number~$i-1$ is never swapped, then~$\pi(i)=\pi_0(i)<i$.
  Otherwise, let~$\tau(e_j)$ be the first swap involving~$i-1$.
  Since~$i-1$ is odd, by definition of~$f$, we have~$\tau(e_j)=(2u,i-1)$ for some~$u < v$.
  Now, if~$2u$ is never swapped again after the~$j$-th swap, then clearly~$\pi(i)=2u<i$.
  Hence, let~$\tau(e_{j'})$, $j' > j$, be the next swap involving~$2u$.
  Since~$2u$ is even and~$e_j > e_{j'}$, we have~$\tau(e_{j'})=(2u,2w-1)$ for some~$u < w < v$. Since the number of swaps is finite, the above arguments can now be repeated to prove the claim.
\end{proof}

\begin{figure}
	\centering
	\begin{tikzpicture}[scale=1.3]
		\draw[curveedge]
			\foreach \i in {1,...,6} {
				(\i, 0) node[vertex] (v\i) {}
			}
			(v1) to (v3)
			(v1) to (v5)
			(v1) to (v6)
			(v2) to (v6)
			(v3) to (v5)
			;
	\end{tikzpicture}
	\caption{A terrain-like graph~$G$. The corresponding Dumont derangement of the second kind is
		$\Pi(G) = 4,1,11,3,9,2,8,7,10,5,12,6$.
		Observe that the isolated vertex corresponds to the pair $(7, 8)$ with $\Pi(G)(7) = \pi_0(7) = 8$ and $\Pi(G)(8) = \pi_0(8) = 7$.
	}
	\label{fig:example}
\end{figure}
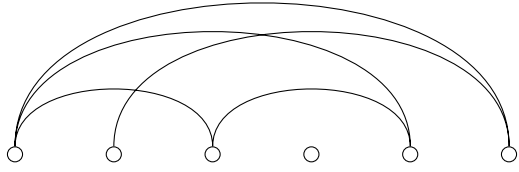

We now continue with the inverse map~$\Gamma$ which maps permutations to graphs.
We first define in which case a pair of an even and an odd number corresponds to an edge.

\begin{definition}
  Let $\pi\in\Dder$ and let $2\le i < j \le 2n-1$, where~$i$ is even and~$j$ is odd.
  Then, $i$ and $j$ are in \emph{edge configuration} if the following is true:
  \[\pi^{-1}(i) < \pi^{-1}(j)\;\iff \; \pi^{-1}(i) \equiv \pi^{-1}(j) \pmod{2}.\]
  Otherwise, $i$ and $j$ are in \emph{non-edge configuration}.
\end{definition}

Observe that $i$ and~$j$ are in edge configuration in~$\pi$ if and only if
they are in non-edge configuration in~$(i,j)\pi$ and vice versa.

\begin{definition}\label{def:inverse}
We define a map~$\Gamma\colon \Dder \to \Gset$ as follows:
Given~$\pi\in\Dder$, and starting with an edgeless graph,
we iterate over all pairs $\{u,v\} \in \binom{[n]}{2}$, $u<v$, ascendingly with respect to some valid order $\leq$.
If $2u$ and $2v-1$ are in edge configuration in~$\pi$, then we insert the edge~$\{u,v\}$ and continue with~$(2u,2v-1)\pi$ instead of~$\pi$.
\end{definition}

\begin{lemma}\label{thm:inverse-welldef}
$\Gamma$ is well-defined.
\end{lemma}
\begin{proof}
We need to show that $\Gamma(\pi)$ does not depend on the chosen order~$\leq$.
Let $\{a, b\}$ and $\{c, d\}$ be incomparable with respect to~$\preceq$.
Then, $2a$, $2b-1$, $2c$, and $2d-1$ are pairwise distinct.
Thus, the two corresponding transpositions $(2a, 2b-1)$ and $(2c, 2d-1)$ commute 
and do not influence whether the respective other pair is in edge configuration.
Therefore, $\Gamma(\pi)$ is not affected by changing the relative order of $\{a, b\}$ and $\{c, d\}$.
\end{proof}

In order to prove that~$\Gamma$ always yields a \tlgraph, we need the following helpful lemma.

\begin{lemma}\label{lem:valid-swap}
  Let~$\pi'$ be the permutation obtained after iterating over~$\{u,v\}$ in the definition of~$\Gamma$. Then, it holds that $\pi'\in\Dder$ and
  $2x$ and~$2y-1$ are in non-edge configuration in~$\pi'$ for all~$x<y$ with~$\{x,y\}\leq\{u,v\}$.
\end{lemma}

\begin{proof}
  Proof by induction. The claim trivially holds before the first iteration.
  
  Let $i\coloneqq 2u$, $j \coloneqq 2v-1$, $p_i\coloneqq\pi^{-1}(i)$, and $p_j\coloneqq\pi^{-1}(j)$.
  Assume that $i$ and~$j$ are in edge configuration in~$\pi$ (otherwise we are done) and let~$\pi'\coloneqq(i,j)\pi$.
  We do a case distinction:

  \textbf{Case 1:} If $p_i < p_j$ and both are even (the case when both are odd is symmetric), then $i < p_i$ (that is, $p_i \ge i+2$) and $j < p_j$.
  
  We first show that $\pi'\in\Dder$.
  Assume for contradiction that $p_i < j$, then we have $i < p_i-1 < j$.
  Since $\{u, p_i/2\}\prec\{u, v\}$, it follows by assumption that $i$ and $p_i-1$ are in non-edge configuration in~$\pi$.
  Thus, $\pi^{-1}(p_i-1)$ is odd if and only if it is larger than~$p_i$, which is a contradiction to $\pi \in \Dder$.
  Therefore $j < p_i$, so we have~$i < j < p_i < p_j$ and $\pi'\in \Dder$.
  
  It remains to check that after swapping~$i$ and~$j$, all previously processed pairs remain in non-edge configuration.
  Clearly, we only need to check pairs involving~$i$ or~$j$ since other pairs are not affected by this swap.
  By assumption, each odd~$j'$ with~$i<j'<j$ was in non-edge configuration with~$i$ in~$\pi$, and thus~$\pi^{-1}(j') < p_i$ since otherwise we obtain a contradiction.
  Similarly, for each even~$i'$ with~$i < i' < j$, we obtain that~$\pi^{-1}(i') < p_i$ or~$\pi^{-1}(i') > p_j$.
  Hence, all considered pairs are still in non-edge configuration in~$\pi'$.

  \textbf{Case 2:} If $p_i > p_j$, then let~$p_j$ be odd and~$p_i$ be even (the reverse case is symmetric), that is,~$j > p_j$ and~$i < p_i$.
  To show that~$\pi'\in\Dder$, first assume for contradiction that~$i < p_j$, that is, we have~$i < p_j+1 < j$ and $\{(p_j+1)/2,v\}\prec\{u,v\}$.
  Hence, $p_j+1$ and~$j$ are in non-edge configuration in~$\pi$ by assumption. Thus,~$\pi^{-1}(p_j+1)$ is odd if and only if it is larger than~$p_j$, which contradicts~$\pi\in\Dder$.
  Therefore,~$i > p_j$ holds. Next, assume for contradiction that~$j > p_i$. Similarly, since~$\{u,p_i/2\}\prec\{u,v\}$, it follows that~$i$ and~$p_i-1$ are in non-edge configuration in~$\pi$, which yields a contradiction. Hence, we obtain that~$p_j < i < j < p_i$ and thus~$\pi'\in\Dder$.

  To check that all previously considered pairs are still in non-edge configuration in~$\pi'$, let $j'$ be odd with~$i< j'<j$ and note that $\pi^{-1}(j') > p_i$ is not possible since this would imply that $\pi^{-1}(j')$ is odd, yielding a contradiction. Thus, $\pi^{-1}(j')$ is even and $p_j < \pi^{-1}(j') < p_i$.
  Analogously, for each even~$i'$ with $i<i'<j$, we obtain that~$\pi^{-1}(i')$ must be odd and~$p_j < \pi^{-1}(i') < p_i$.
  Hence, all considered pairs are in non-edge configuration in~$\pi'$.
\end{proof}

\begin{lemma}
  For every permutation~$\pi\in\Dder$, it holds~$\Gamma(\pi)\in\tlset$.
\end{lemma}

\begin{proof}
  Let~$E$ be the set of edges of~$\Gamma(\pi)$.
  We need to show that~$\Gamma(\pi)$ satisfies the X-property, that is, if~$\{a,b\},\{c,d\}\in E$ with~$a < c < b < d$, then also~$\{a,d\}\in E$.
  Let $\pi_{a,d}$ denote the permutation considered by~$\Gamma$ in the iteration processing $\{a, d\}$ (\Cref{def:inverse}).
  Then, we need to show that, $2a$ and~$2d-1$ are in edge configuration in~$\pi_{a,d}$.
  Without loss of generality, we assume that~$b=\max\{i\mid i < d, \{a,i\}\in E\}$ and $c=\min\{i\mid i > a, \{i,d\}\in E\}$.
  Then, $\{a, b\} \prec \{a, d\}$ and $\{c, d\} \prec \{a, d\}$ and both of these are \emph{covering relations}
  (that is, there is no~$e\in E$ with $\{a, b\} \prec e \prec \{a, d\}$ or $\{c, d\} \prec e \prec \{a, d\}$).
  Moreover, we may choose the valid order used in the definition of~$\Gamma$ such that $\{a, b\} < \{c, d\} < \{a, d\}$ with no other elements which are incomparable to~$\{a,d\}$ with respect to~$\prec$ in between.
  Consequently, we have
  \begin{align*}
    \pi_{a,d}^{-1}(2a) &= \pi_{c,d}^{-1}(2a) = \pi_{a,b}^{-1}(2b-1) \text{ and}\\
    \pi_{a,d}^{-1}(2d-1) &=\pi_{c,d}^{-1}(2c) = \pi_{a,b}^{-1}(2c).
  \end{align*}
  Since $\{c, b\} \prec \{a, b\}$, we know by \cref{lem:valid-swap} that $(2c,2b-1)$ is in non-edge configuration in~$\pi_{a,b}$.
  Thus, $(2a, 2d-1)$ is in edge configuration in $\pi_{a,d}$.
\end{proof}

For a bijection, it remains to show that~$\Pi$ and~$\Gamma$ are inverse of each other.

\section{Proof that~$\Pi$ and~$\Gamma$ are mutually inverse}

In the following, for vertices $1 \leq v < w \leq n$, we say that $v$~is \emph{left of}~$w$ (and $w$~\emph{right of}~$v$).

\subsection{$\Gamma \circ \Pi = \id_{\tlset}$}

We start with the following simple observation.

\begin{observation}\label{thm:remove-edge}
Let $G = (\oneto{n}, E) \in \tlset$ be a \tlgraph and let $\{a, b\} \in E$ with~$a<b$ be a $\preceq$-minimal edge.
Then $G' := G - \{\{a, b\}\}$ is \tl{} and $\Pi(G) = (2a, 2b-1)\Pi(G')$.
\end{observation}
\begin{proof}
  Clearly, deleting the edge~$\{a,b\}$ cannot violate the X-property since~$\{a,b\}$ is $\preceq$-minimal. The second claim follows from choosing a valid edge ordering~$\leq$ where~$\{a,b\}$ is minimal.
\end{proof}

The following lemma will be crucial for showing that~$\Gamma$ yields the correct set of edges.

\begin{lemma}\label{thm:non-edge-config}
  Let~$G = ([n], E) \in \tlset$ be a \tlgraph and $1 \leq x < y \leq n$.
  If there is no edge~$\{u, v\} \in E$ with $\{u, v\} \preceq \{x, y\}$,
  then $(2x, 2y-1)$~is in non-edge configuration in~$\pi \coloneqq \Pi(G)$.
\end{lemma}
\begin{proof}
  We use induction over~$\abs{E}$.
  Note that the claim holds when $E = \emptyset$, as all pairs are in non-edge configuration in $\pi_0 = \Pi(([n], \emptyset))$.
  Furthermore, if $\{u, v\} \in E$ is any $\preceq$-minimal edge with~$u<v$ and $\{2u, 2v-1\} \cap \{2x, 2y-1\} = \emptyset$,
  then, by \cref{thm:remove-edge}, we may delete the edge~$\{u, v\}$ without changing $\pi^{-1}(2x)$ or $\pi^{-1}(2y-1)$.
  Therefore, we assume that $G$~does not contain any such edges.
  
  Let~$e$ now be a~$\preceq$-minimal edge.
  By the above assumption and due to symmetry, we may assume that $e = \{x, z\}$ with $x < y < z$.
  Let $G' := G - \{e\}$ and $\pi' := \Pi(G')$, thus $\pi = (2x, 2z-1)\pi'$ by \cref{thm:remove-edge}.
  
  We first consider the case that~$y$~has no left neighbors.
  Then $\pi^{-1}(2y-1) = \pi_0^{-1}(2y-1) = 2y$.
  If also $z$~has no neighbors left of~$x$, then $\pi^{-1}(2x) = \pi'^{-1}(2z-1) = \pi_0^{-1}(2z-1) = 2z$.
  Thus, $2x$ and~$2y-1$~are in non-edge configuration.
  In the case where $z$~does have some neighbor~$x' < x$ (chosen rightmost),
  we note that $\{x', z\}$ is $\preceq$-minimal in~$G'$ by our assumption above.
  Thus, by induction on $G'' := G' - \{\{x', z\}\} = G - \{\{x, z\}, \{x', z\}\}$, we obtain 
  that $2x'$ and $2y-1$ are in non-edge configuration in~$\Pi(G'')$.
  Since $\pi^{-1}(2x) = \pi'^{-1}(2z-1) = \Pi(G'')^{-1}(2x')$,
  this yields that $2x$ and $2y-1$~are in non-edge configuration in~$\pi$.
  
  It remains to consider the case that~$y$~has a left neighbor~$w < x$ (chosen rightmost).
  Note that, by our assumption, the edge~$\{w, y\}$ is $\preceq$-minimal.
  Due to the X-property, we have $\{w, z\} \in E$.
  Let $x'$ be the rightmost neighbor of~$z$ left of~$x$ and $y'$ the leftmost neighbor of~$w$ right of~$y$.
  Observe that $w \leq x'$ and $y' \leq z$.

  If $w < x'$, then~$\{x',z\}$ is~$\preceq$-minimal in~$G'$ by our assumption above and by the choice of~$w$.
  Thus, by induction on $G''\coloneqq G - \{\{x, z\}, \{x', z\}\}$, we obtain that~$2x'$ and~$2y-1$ are in non-edge configuration in~$\Pi(G'')$.
  Since $\pi^{-1}(2x)=\pi'^{-1}(2z-1)=\Pi(G'')^{-1}(2x')$ and $\pi^{-1}(2y-1)=\Pi(G'')^{-1}(2y-1)$, this yields that~$2x$ and~$2y-1$ are in non-edge configuration in~$\pi$.

  If $y' < z$, then the situation is symmetric to the case~$w < x'$ above and the claim follows by induction on $G - \{\{w, y\}, \{w, y'\}\}$.

  Otherwise, if~$w=x'$ and~$y'=z$, then $\{w,y\}$~is $\preceq$-minimal in~$G'$ and~$\{w,z\}$ is $\preceq$-minimal in~$G''\coloneqq G - \{\{x, z\}, \{w, y\}\}$
  by our assumption and choice of~$w$.
  Thus, by induction on~$G'''\coloneqq G - \{\{x,z\},\{w,y\},\{w,z\}\}$, we obtain that~$2w$ and~$2z-1$ are in non-edge configuration in~$\Pi(G''')$.
  Since $\pi^{-1}(2x) = \pi'^{-1}(2z-1) = \Pi(G''')^{-1}(2w)$ and $\pi^{-1}(2y-1) = \Pi(G'')^{-1}(2w) = \Pi(G''')^{-1}(2z-1)$, it follows that~$2x$ and~$2y-1$ are in non-edge configuration in~$\pi$.
 \end{proof}
\noindent
 For example, one easily verifies that~$(4,5)$ and~$(8,11)$ are in non-edge configuration in \Cref{fig:example}.

 Before finally proving that~$\Gamma$ is the inverse of~$\Pi$, we
 also need the following.
 
\begin{lemma}\label{thm:edge-config}
  Let~$G = ([n], E) \in \tlset$ be a \tlgraph and $\{x, y\} \in E$ a $\preceq$-minimal edge with $x < y$.
  Then $(2x, 2y-1)$~is in edge configuration in~$\pi \coloneqq \Pi(G)$.
\end{lemma}
\begin{proof}
  Let $G' := G - \{\{x, y\}\}$.
  By \cref{thm:remove-edge}, $(2x, 2y-1)\pi = \Pi(G')$.
  By applying \cref{thm:non-edge-config} to~$G'$,
  we see that $(2x, 2y-1)$ is in non-edge configuration in~$(2x, 2y-1)\pi$
  and thus in edge configuration in~$\pi$.
\end{proof}

\begin{lemma}
  For every~$G\in\tlset$, it holds~$\Gamma(\Pi(G))=G$.
\end{lemma}

\begin{proof}
	Let $G = (\oneto{n}, E) \in \tlset{}$ be arbitrary and $\pi := \Pi(G)$.
	We use induction on~$\abs{E}$.
	If~$E = \emptyset$, then the claim holds because all pairs~$(2x, 2y-1)$ with~$1 \leq x < y \leq n$ are in non-edge configuration in~$\pi_0=\pi$.
	
	Now, let $e = \{x, y\} \in E$ with $x < y$ be the minimal edge with respect to some valid ordering~$\leq$ of~$\binom{[n]}{2}$.
	Set $G' := G - \{e\}$ and observe that
	$\Pi(G') = (2x, 2y-1)\pi$ by \cref{thm:remove-edge}.
	Since all pairs $(2u,2v-1)$ with~$\{u, v\} < \{x, y\}$ are in non-edge configuration in~$\pi$ by \cref{thm:non-edge-config},
	and $(2x, 2y-1)$ is in edge configuration in~$\pi$ by \cref{thm:edge-config},
	it holds that
        \[\Gamma(\pi) = \Gamma((2x, 2y-1)\pi) + \{e\} = \Gamma(\Pi(G'))+\{e\}=G' + \{e\} = G,\]
        where the first equality follows from \cref{lem:valid-swap} and the third equality holds by induction.
\end{proof}

\subsection{$\Pi \circ \Gamma = \id_{\Dder}$}

We start with showing that only~$\pi_0$ corresponds to the edgeless graph.

\begin{lemma}\label{thm:pi0}
	The permutation $\pi_0 = (1, 2)(3, 4)\dots (2n-1)(2n)$ is the unique element of~$\Dder$ in which all pairs $i < j$ with $i$~even and $j$~odd are in non-edge configuration.
\end{lemma}
\begin{proof}
	Since $\pi_0^{-1}(i) = i-1$ and $\pi_0^{-1}(j) = j+1$,
	we have $\pi_0^{-1}(i) < \pi_0^{-1}(j)$ and $\pi_0^{-1}(i) \not\equiv \pi_0^{-1}(j) \pmod{2}$, that is, $i$ and $j$ are in non-edge configuration.

        It remains to show that no other permutation in~$\Dder$ has this property.
	To this end, let~$\pi \in \Dder$ be any permutation with that property.
	Clearly, $\pi(2n-1) = 2n$ holds for every permutation in~$\Dder$.
	We claim that $\pi(2n) = 2n-1$.
	
	Suppose for contradiction that $\pi^{-1}(2n-1) < 2n-1$.
	Then $p\coloneqq\pi^{-1}(2n-1)$~must be odd as~$\pi \in \Dder$.
	Consider the set $I := \{i \mid 2 \leq i < 2n-1, i~\text{even}\}$.
	Since $i$ and $2n-1$ are in non-edge configuration in~$\pi$ for all~$i \in I$,
	each $j \in \pi^{-1}(I) := \{\pi^{-1}(i) \mid i \in I\}$ must either be odd or satisfy $j < p$ but not both.
	Thus, we have
	\[
		\pi^{-1}(I) \subseteq \{j \mid 2 \leq j < p,\, j~\text{even}\} \cup \{j \mid p < j \leq 2n-3,\, j~\text{odd}\}.
	\]
	But the left-hand side has cardinality~$\abs{I} = n-1$,
	while the right-hand side has cardinality
	\[
		\frac{p -1}{2} + \frac{2n-3 - p}{2}
		= n -2,
	\]
        which yields a contradiction.
	
	This proves that~$\pi(2n) = 2n-1$.
	Hence, the restriction of~$\pi$ to $\oneto{2n-2}$ is a permutation in~$\tilde{\mathcal{D}}^2_{2n-2}$
	to which we can repeatedly apply the same argument again, ultimately concluding that $\pi = \pi_0$.
\end{proof}

\begin{lemma}
  For every~$\pi\in\Dder$, it holds~$\Pi(\Gamma(\pi))=\pi$.
\end{lemma}
\begin{proof}
	Define $G := (\oneto{n}, E) := \Gamma(\pi)$.
	We use induction on~$\abs{E}$.
	If~$E = \emptyset$, then all pairs $(2x, 2y-1)$ with $1 \leq x < y \leq n$ must be in non-edge configuration in~$\pi$.
	Thus, $\pi = \pi_0 = \Pi(G)$ by \cref{thm:pi0}.
	
	Let now $e = \{x, y\} \in E$ be the minimal edge with respect to some valid ordering~$\leq$ of~$\binom{\oneto{n}}{2}$ and let~$x < y$.
	By choice of~$e$ and by definition of~$\Gamma$, the pair $(2x, 2y-1)$ is in edge configuration in~$\pi$ and any pair~$(2u, 2v-1)$ with $\{u, v\} < \{x, y\}$ is in non-edge configuration.
	Define $G' := G - \{e\}$ and $\pi' := (2x, 2y-1)\pi$.
	Thus, $G' = \Gamma(\pi')$ by definition of~$\Gamma$ and by~\cref{lem:valid-swap}.
	By induction, we have $\Pi(G') = \pi'$, which yields $\Pi(G) = (2x, 2y-1)\Pi(G') = (2x, 2y-1)\pi' = \pi$.
\end{proof}

\section{Conclusion}

We close with some natural follow-up questions that are left open.
\begin{compactitem}
\item Which Dumont permutations correspond to the subset of persistent graphs (or even terrain-visibility graphs)? Is there a simple characterization?
  
\item Which graphs do we obtain if we reverse the order of vertex pairs in the definition of~$\Gamma$ (\Cref{def:inverse}) from outer to inner?
Is there a relation to \emph{non-jumping} graphs~\cite{AFS19,AFKS22}?

\item Is there an easy bijection between terrain-like graphs and \emph{alternation acyclic tournaments} (which are also counted by the median Genocchi numbers~\cite{Hetyei19})?
  
\end{compactitem}

\printbibliography

\end{document}